\documentclass[reqno,a4paper,12pt]{amsart}
\usepackage{amsmath,amssymb,amsthm,amscd,amsfonts,amsbsy,bm}
\usepackage[mathscr]{eucal}

\numberwithin{equation}{section}

\renewcommand{\a}{\alpha}
\renewcommand{\b}{\beta}
\newcommand{\g}{\gamma}
\newcommand{\G}{\Gamma}
\renewcommand{\d}{\delta}
\newcommand{\D}{\Delta}
\newcommand{\e}{\epsilon}

\newcommand{\z}{\zeta}
\newcommand{\y}{\eta}
\renewcommand{\th}{\theta}

\renewcommand{\k}{\kappa}

\renewcommand{\l}{\lambda}
\renewcommand{\L}{\Lambda}
\newcommand{\m}{\mu}
\newcommand{\n}{\nu}
\newcommand{\x}{\xi}

\renewcommand{\r}{\rho}
\newcommand{\s}{\sigma}

\newcommand{\vs}{\varsigma}
\renewcommand{\t}{\tau}

\newcommand{\vf}{\varphi}

\renewcommand{\o}{\omega}
\renewcommand{\O}{\Omega}

\newcommand{\C}{{\mathbb C}}
\newcommand{\R}{{\mathbb R}}
\newcommand{\Z}{{\mathbb Z}}

\newcommand{\ab}{{\mathbf a}}

\newcommand{\db}{{\mathbf d}}

\newcommand{\Ab}{{\mathbf A}}

\newcommand{\Hb}{{\mathbf H}}
\newcommand{\Ib}{{\mathbf I}}

\newcommand{\Rb}{{\mathbf R}}

\newcommand{\Bc}{{\mathcal B}}

\newcommand{\Dc}{{\mathcal D}}

\newcommand{\Hc}{{\mathcal H}}

\newcommand{\Lc}{{\mathcal L}}

\newcommand{\Pc}{{\mathcal P}}

\newcommand{\Rc}{{\mathcal R}}

\newcommand{\supp}{\hbox{{\rm supp}}\,}

\DeclareMathOperator{\re}{{\rm Re}\,}

\newcommand{\loc}{\operatorname{loc}}

\newtheorem{theorem}{Theorem}[section]
\newtheorem{proposition}[theorem]{Proposition}
\newtheorem{lemma}[theorem]{Lemma}

\newtheorem*{theorem*}{Theorem}

\theoremstyle{definition}
\newtheorem{definition}[theorem]{Definition}

\theoremstyle{remark}
\newtheorem{remark}[theorem]{Remark}

\newcommand{\sasha}[1]{} 
\newcommand{\GR}[1]{}

\begin{document}

\title[Lower eigenvalue bounds]{On  lower eigenvalue bounds for Toeplitz operators with radial symbols in Bergman spaces }

\author[G. Rozenblum]{Grigori Rozenblum}
\address{1. Department of Mathematics \\
                          Chalmers University of Technology \\
                          2.Department of Mathematics  University of Gothenburg \\
                          Chalmers Tv\"argatan, 3, S-412 96
                           Gothenburg
                          Sweden}
\email{grigori@math.chalmers.se}

\begin{abstract}
We consider Toeplitz operators in different Bergman type spaces, having radial symbols with variable sign. We show that if the symbol has compact support or decays rapidly, the eigenvalues of such operators cannot decay too fast, essentially faster than for a sign-definite symbol with the same kind. On the other hand, if the symbol decays not sufficiently rapidly, the eigenvalues of the corresponding operator may decay faster than for the operator corresponding to the  absolute value of the symbol.
\end{abstract}
\date{}

\maketitle


\section{Introduction}\label{intro}

Toeplitz operators arise in many fields of Analysis. The general setting is the following. Let $\Hc$ be a Hilbert space of functions and  $\Bc$ be  a closed subspace in $\Hc$. For a function $V$, called the \emph{symbol} further on, the Toeplitz operator $T_V:\Bc\to\Bc$ acts as $T_V: u\mapsto PVu$, where $P$ is the orthogonal projection from $\Hc$ onto $\Bc$. Of course, it is supposed that the operator of multiplication by $V$ maps $\Bc$ into $\Hc$.

  In the present paper we consider Toeplitz operators in some Bergman type spaces. Let $\O$ be a domain in $\R^d$, or $\C^d$,  $\Hc$ be the space $L^2(\O)$ with respect to some measure  $\m$ and $\Bc$ be the subspace in $\Hc$ consisting of solutions of some elliptic equation or system. The leading example here is provided by   Bergman-Toeplitz operators, where $\O$ is a bounded domain with nice boundary and $\Bc$ consists of harmonic functions in $\Hc$ (the harmonic Bergman space), and, in the complex case,
  $\Bc$ consists of analytical functions in $\O$ (the analytical Bergman space). Another series of  examples is given by Bargmann-Toeplitz operators, where $\O$ is the whole (real or complex) space, $\Hc=L^2(\O)$ with respect to the Gaussian measure and $\Bc$ consists of harmonic or (in the complex case) analytical functions in $\Hc$.

We are interested in the spectral properties of Toeplitz operators for the case when the symbol $V$, which is supposed to be real and bounded, has compact support  (when $\O$ is a ball) or decays rapidly at infinity (when $\O$ is the whole space). One can easily see that such operator is compact, and our question is about determining how fast the eigenvalues of $T_V$ tend to zero. The interest to this topic grew recently due to the close relation of the spectral properties of Toeplitz operators to the spectral analysis  of the perturbed Landau Hamiltonian describing the quantum particle in a homogeneous magnetic field.

For a sign-definite symbol, in the complex Bargmann case, rather complete results were obtained in \cite{RaiWar}, \cite{MelRoz} and, in dimension $d=1$, improved in \cite{FilPush}, see also references therein. Even earlier, the case of complex Bergman spaces  in dimension $d=1$ has been studied in  \cite{Parf}. It was proved that the eigenvalues of the Toeplitz operator follow an asymptotic law, of an exponential type  for the Bergman case and super-exponential type for the Bargmann one.

For the case of the symbol $V$  having variable sign,  it was for a long time unclear, even whether it is possible that the positive and negative parts $V_\pm$ of $V$ can compensate each other almost completely, so that the spectrum of $T_V$ is finite. It has been proved only recently, see \cite{Lue2}, that such complete cancelation is impossible, in other words, for a nontrivial symbol $V$ with compact support  the Toeplitz operator cannot have finite rank; see \cite{RozToepl} for the most complete
 results on the finite rank problem and related references. A further analysis in  \cite{PushRoz2}  has shown that the (infinite now) spectrum of the Toeplitz operator depends essentially on the geometry of the support of $V_\pm$. In particular, if, say, the support of $V_+$ surrounds the support of $V_-$ (in a proper meaning) then the negative spectrum of $T_V$ is finite and the asymptotic behavior  of  the positive eigenvalues is the same as if $V_-$ were absent.  On the other hand, if $V_\pm$ are supported in geometrically well separated sets, then both the positive and the negative spectra of $T_V$ are infinite, and, taken together, obey the same asymptotic law as if $V$ were sign-definite (in the Bargmann case); these results can be understood that no cancelation of $V_\pm$ takes place for this class of symbols, as it concerns the properties of the spectrum.

In the present paper we continue the study of the spectrum of Toeplitz operators with non-sign-definite symbol. We consider the
model case of the symbol $V$ being radial, i.e., depending only on the the distance to the origin; no  restrictions on  the supports of $V_\pm$ are imposed.  We find out that there is essential difference in the spectral properties of Toeplitz operators with rapidly decaying symbols (including compactly supported ones, the exact definitions are given in the paper), on the one hand, and  symbols decaying not that rapidly, on the other. In  the former case we establish that, although no information on the positive and negative spectra of $T_V$ separately can be obtained,  the distribution function of the positive and negative eigenvalues counted together, i.e. of  singular numbers of the operator,  is subject to lower asymptotical bounds that have the same order and even the same coefficient as if the symbol were sign definite. This is expressed, in a general form, by the relation \eqref{EstimGeneral}. So, again, no cancelation happens.  On the other hand, in the latter case it is possible that the eigenvalues of the Toeplitz operator decay considerably faster than for the operator with the corresponding sign-definite symbol.

We start in Section \ref{Sect2}  with describing the Bergman type spaces under consideration (we consider operators in the Bergman and Bargmann spaces of analytical and harmonic functions as well as in the spaces of solutions of the Helmholtz equation) and finding the expression for the eigenvalues. These expressions are quite explicit. The asymptotic formulas for eigenvalues are obtained, first for $V$ being the characteristic function of a ball, and then these results are carried over to general sign-definite compactly supported  case by means of simple monotonicity arguments. For Bargmann spaces a class of symbols with rapid decay at infinity is  considered as well.
Some of these results are well known, the remaining are obtained in a more or less standard way -- we, however, present them  all here for further reference.   In the end of  Section \ref{Sect2} we show that, similarly to the results in \cite{PushRoz2}, the same asymptotics holds even for non-sign-definite symbols, as long as geometrically the support of $V_+$ surrounds the support of $V_-$ (or the other way round).

Passing to  general  non-sign-definite radial symbols,   we encounter a serious inconvenience. We still can write the explicit expression for the eigenvalues, however   the numbering of the eigenvalues in the non-increasing or the
non-decreasing order does not coincide with their natural numbering, stemming from  the one in the  separation of variables, and  the relation between two numberings is rather hard to control. To handle this circumstance, we need  certain considerations from infinite combinatorics (Proposition \ref{Prop.reord.2}).
In order to prove that the eigenvalues cannot decay too fast, we need rather advanced results in complex analysis, and these results are also presented in Sect \ref{Sect3}. The main results of the paper on the lower eigenvalue estimates for general non-sign-definite radial symbols are presented in Sect.\ref{Sect4}, see Theorems \ref{ThMainBergmanC}, \ref{ThEstBargC}. Finally, in Sect.\ref{Sect5}, we describe  examples showing that a considerable cancelation may take place for not sufficiently rapidly  decaying fast oscillatory symbols.
\section {Eigenvalues of Toeplitz operators with  radial weight}\label{Sect2}

In this Section we calculate the eigenvalue asymptotics for Toeplitz  operators with radial symbols in the spaces under consideration. For sign-definite symbols some of these results are known, others are obtained in a standard way using the explicit expression for eigenvalues.
Further on, we extend these results to a class of symbols, not necessarily sign-definite, but having a constant sign  at the periphery of the support.
\subsection{Eigenvalues and  re-ordering -- 1}\label{SSreordering1}
For each of operators $T_V$ under consideration in the paper,  we are going to find explicitly the sequence of eigenvalues $\L_k$ having multiplicities $\db_k$. To describe the ordered set of eigenvalues, counting multiplicities, one should consider the set of the numbers $\L_k$, each $\L_k$ counted $\db_k$ times, and then re-order  the positive numbers in the non-increasing way and the negative ones in the nondecreasing way. Thus we obtain two (finite or infinite) sequences  $\l_n^{\pm}$ of eigenvalues of $T_V$. The union of the the sequences $\pm\l_n^{\pm}$ is the sequence of $s$-numbers, numerated in the nonincreasing order $s_n=s_n(T_V)$.

For problems under consideration, it is often more convenient to describe the spectrum by means of  the counting functions defined as
\begin{equation}\label{counting}
    n_\pm(\l)=\#\{n:\pm\l_n^{\pm}>\l\}, \ n(\l)=n_+(\l)+n_-(\l)=\#\{n:s_n>\l\};
\end{equation}
we will include the designation of the operator and the space  in question in the notation, when needed.
In the terms of multiplicities, we, obviously, have
\begin{equation}\label{counting multi}
    n_\pm(\l)=\sum_{\pm\L_k>\l}\db_k, \ n(\l)=\sum_{|\L_k|>\l}\db_k.
\end{equation}

In most simple cases below, the sequence $\L_k$ is non-negative and already non-increasing, and thus no re-ordering is needed. However, generally, we should not expect that the sequence $|\L_k|$ is monotonous, and thus the question arises, how the estimates for $|\L_k|$ are related to estimates of this sequence monotonously re-ordered.  In one direction, the result is obvious. We, however, formulate it in order to be able to refer to it later on.
\begin{proposition}\label{PropReorderTriv}Let $a_k, b_k, \ k=0,1,\dots ,$ be two sequence of real numbers, so that $b_k>0,$ $b_k\to 0$ monotonously. Suppose that $|a_k|\le b_k$. Denote by $a^*_k$ the sequence obtained by the non-increasing re-ordering of the sequence $|a_k|$. Then $a^*_k\le b_k.$
\end{proposition}
Of course, generally, one should not expect the direct conversion of Proposition \ref{PropReorderTriv} to be correct: an estimate for the monotonously re-ordered sequence cannot be carried over to  the initial sequence.
It turns out, however, that in a certain sense Proposition \ref{PropReorderTriv} can be partially conversed, see Proposition \ref{prop.reord.major}.

The operators we consider in this paper have very fast, exponential or even super-exponential rate of decay of eigenvalues. Since these eigenvalues have very high multiplicity, the eigenvalues $\l_n^{\pm}$, counting multiplicity, do not follow any regular asymptotic law, unlike the well-studied case  of elliptic operators. Therefore it is more convenient to consider the behavior of the eigenvalues in the logarithmic scale, where the oscillations caused by high multiplicities are suppressed and a regular asymptotics exists.  In this scale, in particular, the leading term of the asymptotics does not change when the symbol $V$ (and thus the eigenvalues of the Toeplitz operator) is multiplied by a positive constant.  Alternatively, the asymptotical behavior of eigenvalues can be described by their counting function. Unlike the case of power-like asymptotics, typical for elliptic boundary problems, the asymptotic formula for the counting function is not equivalent to the one for the eigenvalues, however it is equivalent to the asymptotic eigenvalue formula in the logarithmic scale. We will  use this equivalence persistently.

We will use the following notation. For functions $f$ and $g$ of a real or integer argument $t$, the symbol $f(t)\sim g(t)$ means, as usual, $f(t)/g(t)\to 1$ as $t\to\infty$ or $t\to 0$, which is always clear from the context. The relation $f(t)\lesssim g(t)$ means that $\limsup f(t)/g(t)\le1$, the
 obvious meaning has the notation $f(t)\gtrsim g(t)$. Finally, $f(t)\asymp g(t)$ is used when $c f(t)\le g(t)\le C f(t)$ for sufficiently large (or small) $t$ and for some positive constants $c,C.$

\subsection{Bergman type spaces and quadratic forms}
\subsubsection{The spaces.}
The following Bergman type spaces will be considered in this paper.
\begin{enumerate}
\item Bergman spaces. \begin{itemize}
\item The Bergman spaces of analytical functions in the ball. Let $D^{2d}\subset\C^d, d\ge1$ be the  ball with radius $\Rb$. The space   $\Hc$ is the space $L^2(D^{2d})$ with respect to the Lebesgue measure  and $\Bc=\Bc_\Rb^{\C}\subset\Hc$ consists of analytical functions.
    \begin{remark}\label{Another Weight} In the literature,  Bergman spaces with the Lebesgue measure with weight $(1-(|z|/\Rb)^2)^\a$ are considered as well. The results of the paper extend to this case almost automatically.
    \end{remark}
\item The Bergman spaces of harmonic functions in the ball.  Let $D^d\subset \R^d, d>1,$ be the ball with radius $\Rb$.  The space $\Hc$ is $L^2(D^d)$ with respect to the Lebesgue measure and $\Bc=\Bc_\Rb^{\R}\subset\Hc$ consists of harmonic functions.
\item The Bergman spaces of solutions of the Helmholtz equation. The space $\Hc$ is, again, $L^2(D^d)$ and $\Bc =\Bc_\Rb^{\Hb}\subset\Hc$ consists of solutions of the Helmholtz equation $\D u+ u=0$.\end{itemize}
\item  Bargmann spaces.
\begin{itemize}
\item The Bargmann spaces of analytical functions in $\C^d$. Here $\Hc$ is $L^2(\C^d)$ with respect to the Lebesgue measure with Gaussian weight $e^{-|z|^2}$ and $\Bc=\Bc^\C\subset\Hc$ consists of analytical functions.
\item The Bargmann spaces of harmonic functions in $\R^d$. Here $\Hc$ is $L^2(\R^d)$ with respect to the Lebesgue measure with Gaussian weight $e^{-|x|^2}$ and $\Bc=\Bc^\R\subset\Hc$ consists of harmonic functions.
\item The Bargmann spaces of solutions of the Helmholtz equation in $\R^d$. Here     $\Hc$ is $L^2(\R^d)$ with respect to the Lebesgue measure with Gaussian weight $e^{-|x|^2}$ and $\Bc=\Bc^{\Hb}\subset\Hc$ consists of solutions of the Helmholtz equation.
    \end{itemize}
\item The Agmon-H\"ormander space $\Hc=B^*$, see \cite{AgmonHorm}, \cite{Strich}, is defined as consisting of (equivalence classes of) functions $u\in L^2_{\loc}(\R^d)$ such that the norm
    \begin{equation}\label{AHspace}
         \|u\|_{B^*}=\left(\sup_{r\in(0,\infty)}r^{-1}\int_{|x|<r}|u|^2dx\right)^{\frac12}
    \end{equation}
    is finite. All functions $u$ in this space, satisfying the Helmholtz equation, form a closed subspace $\Bc=\Bc^{\Ab\Hb}$. For $u\in \Bc^{\Ab\Hb}$, the limit     \begin{equation}\label{AKlim}
|||u|||^2=\lim_{r\to\infty}r^{-1}\int_{|x|<r}|u|^2dx
\end{equation}
exists, and it defines the norm equivalent to \eqref{AHspace} (see \cite{Strich}, Lemma 3.2.)
\end{enumerate}
\subsubsection{Toeplitz operators and quadratic forms}
The convenient way to study the eigenvalues of Toeplitz operators is by using the quadratic form setting. Let $\Bc\subset\Hc$ be the Bergman type space under study and $<\cdot,\cdot>$ and $\|\cdot\|$ be the corresponding
scalar product and norm. The Toeplitz operator $T_V$ in $\Bc$ is defined by the quadratic form $h_V[u]=<Vu,u>,\  u\in\Bc $. It is convenient to use this definition even in the case when one does not consider the embracing space, as, for example, for $\Bc=\Bc^{\Ab\Hb}$: having a Bergman type space $\Bc$ we will still call the operator $T_V$ defined by the quadratic form $\int V|u|^2 d\m$ in $\Bc$ the Toeplitz
operator in $\Bc$ with symbol $V$. As soon as  a complete system of functions $u_n\in \Bc$ is found, which diagonalizes both quadratic forms $\|u\|^2$ and $h_V[u]$, this system can serve as a complete system of eigenfunctions of $T_V$, with eigenvalues $h_V[u_n]/\|u_n\|^2$. We emphasize here again that these eigenvalues should be properly re-ordered.

In this paper we are going to study Toeplitz operators with radial symbols, i.e., $V(z)=V(|z|)$ in the complex case and $V(x)=V(|x|)$ in the real case. For such symbols the eigenfunctions and eigenvalues of the Toeplitz operator can be found explicitly by means of passing to spherical co-ordinates.

\subsection{Operators in Bergman spaces}
 In the  sections to follow we collect the results on the eigenvalue asymptotic formulas for Toeplitz operators in Bergman spaces. Some of them are known, the rest are obtained in a standard way, using the explicit  expressions for the eigenvalues.
 \subsubsection{Complex Bergman spaces.}\label{ss.cBergman} Denote by $\Pc^\C_k$ the space of homogeneous polynomials of degree $k$ of variables $z_1,\dots, z_d$. It has dimension $\db_k^{\C}=\binom{k+d-1}{d-1}= \frac{k^{d-1}}{(d-1)!}(1+O(k^{-1}))$.  In the space of  functions of the form $Z(\o)=p(z)|z|^{-k}$, $p\in\Pc^\C_k$, $\o=z|z|^{-1}\in S^{2d-1}$, we choose a basis $Z_{k,j}(\o)$, $j\in[1,\db_k^\C]$, orthonormal with respect to the Lebesgue measure on the sphere $S^{2d-1}$ (complex spherical functions). The functions $u_{k,j}(z)=|z|^kZ_{k,j}(\o), \ k=0,1,\dots, \ j\in[1,\db_k^\C],$ form an orthogonal  basis in the space $\Bc_\Rb^{\C}$. For a radial function $V(|z|)$, this system of functions diagonalize also the quadratic form $\int V(|z|)|u(z)|^2 d\m$. Therefore, the functions $u_{k,j}$ form a complete system of eigenfunctions of the Toeplitz operator $T_V$ in $\Bc_\Rb^{\C}$ with eigenvalues
\begin{equation}\label{Bergman C}
    \L_k=\L_k^\C(V)=\frac{<Vu_{k,j},u_{k,j}>}{\|u_{k,j}\|^2_{\Bc_{\Rb}^{\C}}}=(2k+2d)\Rb^{-(2k+2d)}\int_0^\Rb V(r)r^{2k+2d-1}dr,
\end{equation}
having multiplicity $\db_k^{\C}$. By re-ordering, see Sect.\ref{PropReorderTriv}, we obtain two sequences  $\l_n^{\pm}$ of eigenvalues of $T_V$. By the results of \cite{AlexRoz} (see also  \cite{RozToepl}), at least one of these sequences is infinite.

Let $V$ be the characteristic function of the ball $D_b$ with center at the origin and radius $b\in(0,\Rb)$. Then there are no negative eigenvalues, and the positive eigenvalues, by \eqref{Bergman C}, are $\L_k=(b/\Rb)^{2k+2d}$ with multiplicities $\db_k^{\C}$. Taking into account the asymptotics $\db_k^{\C}\sim k^{d-1 }/(d-1)!$, we have
in terms of the counting function,
\begin{equation}\label{asymp.compl.count}
    n(\l; T_V,\Bc^\C_\Rb)\sim \sum\limits_{(b/\Rb)^{2k+2d}>\l}\frac{k^{d-1}}{(d-1)!}\sim (d!)^{-1}(2|\log(b/\Rb)|)^{-d}|\log\l|^d, \ \l\to 0,
\end{equation}
or,  in the logarithmic scale,
\begin{equation}\label{asymp.compl}
    \log(\l_n^+)=\log(s_n(T_V))\sim2 (nd!)^{\frac1d}\log (b/\Rb).
\end{equation}

\subsubsection{Harmonic Bergman spaces}\label{ss.hBergman}
Denote by $\Pc^\R_k$ the space  of degree $k$ homogeneous  harmonic polynomials of the variables $x_1,\dots, x_d$. This space has dimension $\db_k^\R=\binom{d+k-1}{d-1}-\binom{d+k-2}{d-2}= \frac{2}{(d-2)!}k^{d-2}(1+O(k^{-1}))$ (see, e.g., calculations in \cite{Shubin}, Sect.22). In the space of  functions of the form $Y(\o)=p(x)|x|^{-k}$, $p\in\Pc^\R_k$, $\o=x|x|^{-1}\in S^{d}$, we choose a basis $Y_{k,j}(\o)$, $j\in[1,\db_k^\R]$, orthonormal with respect to the Lebesgue measure on the sphere $S^{d-1}$, i.e., the usual spherical functions. The functions $u_{k,j}(x)=|x|^kY_{k,j}(\o), k=0,1,\dots, \ j\in[1,\db_k^\R],$ form an orthogonal  basis in the space $\Bc_\Rb^{\R}$. For a radial function $V(|x|)$, this system of functions diagonalizes also the quadratic form $\int V(|x|)|u(x)|^2 d\m$. Therefore, the functions $u_{k,j}$ form a complete system of eigenfunctions of the Toeplitz operator $T_V$ in $\Bc_\Rb^{\R}$ with eigenvalues $\L_k^\R$ given by
\begin{equation}\label{Bergman R}
     \L_k=\L_k^\R(V)=(2k+d)\Rb^{-2k-d}\int_0^\Rb V(r)r^{2k+d-1}dr
\end{equation}
and multiplicities $\db_k^\R$. Taking into account the  multiplicities and reordering, as in Section \ref{ss.cBergman}, we obtain the eigenvalue sequences $\l_n^{\pm}$ and the sequence of $s-$numbers $s_n$.
Again, as it is shown in \cite{AlexRoz},  the sequence $s_n$ and at least one of the sequences $\l_n^{\pm}$   are infinite.

For $V$ being the characteristic function of the ball $D_b$, the eigenvalues $\L_k^\R$ are equal to $\L_k^\R=(b/\Rb)^{2k+d}$, by \eqref{Bergman R}. Thus, there are no negative eigenvalues $\l_n^-$, while for the positive eigenvalues $\l_n^+$ we have the asymptotics
\begin{equation}\label{asymp.harm}
    n(\l; T_V,\Bc^\R_\Rb)\sim \sum\limits_{(b/\Rb)^{2k+d}>\l}2\frac{k^{d-2}}{(d-2)!}\sim 2((d-1)!)^{-1}(2|\log(b/\Rb)|)^{-d+1}|\log\l|^{d-1}, \ \l\to 0,
\end{equation}
and, in the logarithmic scale,

\begin{equation}\label{asympt.Harm.l}
    \log\l_n^+\sim 2\log(b/\Rb)((d-1)!/2)^{\frac{1}{d-1}}n^{\frac{1}{d-1}}, \ n\to\infty.
\end{equation}

\subsubsection{Helmholtz Bergman spaces}\label{ss.HelmBergman}
After passing to spherical co-ordinates in the Helmholtz equation,
we arrive at the orthogonal system of functions
 \begin{equation}\label{SystemHelm}
 u_{k,j}(x)=Y_{k,j}(\o)|x|^{-\frac{d-2}{2}}J_{k+\frac{d-2}{2}}(|x|); \ \o=x|x|^{-1}\in S^{d-1}, k = 0,1,\dots, \ j=1,\dots, \db_k^\R,
 \end{equation}
 where $J_\n(r)$ are the Bessel functions and $Y_{k,j}$ are the real spherical functions as in Sect.\ref{ss.hBergman}. For a radial symbol $V(|x|)$, the eigenvalues of the Toeplitz operator equal
 \begin{equation}\label{EigenHelm}
    \L_k=\L_k^{\Hb}(T_V)=\frac{\int_0^\Rb V(r)J^2_{k+\frac{d-2}{2}}(r)rdr}{\int_0^\Rb J^2_{k+\frac{d-2}{2}}(r)rdr},
 \end{equation}
 with multiplicity $\db_k^\R$.
 The integral in the denominator in \eqref{EigenHelm} is estimated by means of the identity (see, e.g., \cite{Watson})
 \begin{equation}\label{BesselIntegr}
    \int_0^R J^2_\n(r) r dr=\frac{R^2}{2}[J^2_\n(R)-J_{\n-1}(R)J_{\n+1}(R)],
 \end{equation}
 and the asymptotics (see, again \cite{Watson}), uniform in $r$ on any finite interval $[a,b]\subset[0,\infty)$:
\begin{equation}\label{BesselAs}
  J_\n(r)\sim\left(\frac{r^2}{2}\right)^{\n}(\G(\n+1))^{-1}, \ |\n|\to +\infty, \re \n\ge0.
\end{equation}
So, we obtain
\begin{equation}\label{BesselintAs}
    \int_0^\Rb J^2_{k+\frac{d-2}{2}}(r)rdr\sim \left(\frac{\Rb^2}2\right)^{k+d/2}\frac{1}{\G(k+\frac{d}{2})\G(k+\frac{d+2}{2})}, \ k\to \infty.
\end{equation}

Again, as before, the numbers $ \L_k^{\Hb},$ counted with multiplicities $\db_k^\R$ and properly re-ordered, form the sequences $\l_n^\pm=\l_n^\pm(T_V)$ of eigenvalues of $T_V$, and the union of the sequences $\pm\l_n^\pm$ is the sequence of $s$-numbers $s_n=s_n(T_V)$. It is proved in \cite{RozToepl} that for $d>2$ at least one of the sequences $\l_n^\pm$ is infinite. The proof in \cite{RozToepl} does not cover the case $d=2$,  and the above infiniteness will follow from the results of the present paper.

For $V$ being the characteristic function of the ball $|x|\le b<\Rb$, the numbers $\L_k^{\Hb}(V)$ have, by \eqref{BesselIntegr}, \eqref{BesselAs}, and \eqref{BesselintAs}, the asymptotics
\begin{equation}\label{HelmCharAs}
    \L_k^{\Hb}\sim (b/\Rb)^{2k+d }.
\end{equation}
Therefore, taking into account multiplicities,
the eigenvalues $\l_n^+$ obey the asymptotic law \eqref{asymp.harm}, \eqref{asympt.Harm.l}, the same  as for the harmonic Bergman space.
\subsection{Operators in Bargmann and AH spaces}

\subsubsection{Complex Bargmann spaces}\label{ss.cBargmann} The functions $u_{k,j}(z)=Z_{k,j}(\o)|z|^k, $  $j\in[1,\db_k^\C],$ form an orthogonal basis in the Bargmann space $\Bc^{\C}.$ Thus, the eigenvalues of the Toeplitz operator $T_V$ in $\Bc^{\C}$ equal
\begin{equation}\label{CBargm}
    \L_k^\C
    =\frac{\int_0^\infty V(r)r^{2k+2d-1}e^{-r^2}dr}{\int_0^\infty r^{2k+2d-1}e^{-r^2}dr}=2\frac{\int_0^\infty V(r)r^{2k+2d-1}e^{-r^2}dr}{\G(k+d)}.
\end{equation}
For the case of $V(r)$ being the  characteristic function of $D_b$, $0<b<\infty$, obviously,
$$\left|\log\int_0^\infty V(r)r^{2k+2d-1}e^{-r^2}dr\right|\asymp k,$$
and, therefore, by  the Stirling formula,
\begin{equation}\label{BargLambda}
   | \log \L_k^{\C}|\sim k\log k.
 \end{equation}
Taking into account the multiplicities, we obtain
for the eigenvalues of $T_V$:
\begin{equation}\label{CBargm.eigenvaluesN}
    n(\l)=\sum_{\L_k^\C>\l}\db_k^\C\sim \frac1{d!}\left(\frac{|\log \l|}{\log|\log \l|}\right)^d,
\end{equation}
or, in the logarithmic scale, inverting \eqref{CBargm.eigenvaluesN}:
\begin{equation}\label{CBargm.eigenvalues}
    \log(\l_n^+)=\log(s_n)\sim - d^{-1}(d!)^{\frac1d} n^{\frac1d}\log n.
\end{equation}
\begin{remark}\label{remWithout b}The asymptotic relation \eqref{CBargm.eigenvaluesN} was found  in \cite{MelRoz} (in \cite{RaiWar} for $d=1$);  it was discovered there, in particular, that the leading term in the eigenvalue asymptotics  of  Bargmann-Toeplitz operators does not depend on the symbol $V\ge 0$ with compact support (of course, provided it is not  identically zero). In \cite{FilPush}, for $d=1$, the second term of the asymptotics in \eqref{CBargm.eigenvalues} was found, depending on the logarithmic capacity of $\supp V$.
\end{remark}

\subsubsection{Harmonic Bargmann spaces}\label{ss.hBargmann}
The functions $u_{k,j}(x)=Y_{k,j}(\o)|x|^k, $  $j\in[1,\db_k^\R],$ form an orthogonal basis in the Bargmann space $\Bc^{\R}.$ Thus, the eigenvalues of the Toeplitz operator $T_V$ in $\Bc^{\R}$ equal
\begin{equation}\label{RBargm}
    \L_k^\R
    =\frac{\int_0^\infty V(r)r^{2k+d-1}e^{-r^2}dr}{\int_0^\infty r^{2k+d-1}e^{-r^2}dr}=2\frac{\int_0^\infty V(r)r^{2k+d-1}e^{-r^2}dr}{\G(k+\frac{d}{2})}.
\end{equation}
For  $V$ being the characteristic function of the interval  ball $D_b$ we obtain for the eigenvalues of $T_V$, taking into account the multiplicities:
\begin{equation}\label{RBargm.eigenvalues}
    n(\l)=\sum_{\L_k^\R>\l}\db_k^\R\sim2((d-1)!)^{-1}\left(\frac{|\log\l|}{\log|\log\l|}\right)^{d-1},
\end{equation}
or, in the logarithmic scale, inverting \eqref{RBargm.eigenvalues},
\begin{equation}\label{RBargm.eigenvaluesL}
    \log(\l_n^+)=\log(s_n)\sim -(d-1)^{-1}((d-1)!/2)^{\frac{1}{d-1}}(n^{\frac1{d-1}}\log n), \ n\to\infty.
\end{equation}

\subsubsection{Helmholtz Bargmann spaces}\label{ss.HelmBargmann}
The functions
\begin{equation*}
u_{k,j}(x)=Y_{k,j}(\o)|x|^{-\frac{d-2}{2}}J_{k+\frac{d-2}{2}}(|x|); \  k = 0,1,\dots, \ j=1,\dots, \db_k^\R,
\end{equation*}
form an orthogonal basis in the space $\Bc^{\Hb}.$ Thus, the eigenvalues of $T_V$ in $\Bc^{\Hb}$ equal
\begin{equation}\label{HarmBargm}
 \L_k=\L_k^{\Hb}(V)=\frac{\int_0^\infty V(r)J_{k+\frac{d-2}{2}}(r)^2re^{-r^2}dr}{\int_0^\infty J_{k+\frac{d-2}{2}}(r)^2 e^{-r^2}rdr},
\end{equation}
with multiplicity $\db_k^\R.$ The denominator in \eqref{HarmBargm} equals $\frac12\exp(-\frac12)I_{k+\frac{d-1}{2}}(\frac12)$, where $I_\n$ is the modified Bessel function (see \cite{Grad}, 6.663.2).  By \eqref{BesselAs}, this denominator has the asymptotics $\left(\frac{1}2\right)^{k+\frac{d-1}{2}}\exp(-\frac12)\G(k+\frac{d+1}{2})^{-1}.$

For $V$ being the characteristic function of the ball $D_b$, the numerator in \eqref{HarmBargm}  is estimated from above and from below by constants times $b^{2k+d}\left(\G(k+\frac{d}{2})\G(k+\frac{d+1}{2})\right)^{-1}.$
Therefore, in this case,  the eigenvalues $\L_k^{\Hb}$ obey two-sided asymptotic estimates
 \begin{equation}\label{L Helm Barg}
 \L_k^{\Hb}\asymp \left(\frac{1}2\right)^{k+\frac{d}{2}}b^{2k+d}(\G(k+\frac{d+1}{2}))^{-1}.
\end{equation}
 Taking into account multiplicities, the eigenvalues of the Toeplitz operator $T_V$ in the space $\Bc^\Hb$ have the same asymptotics \eqref{RBargm.eigenvalues}, \eqref{RBargm.eigenvaluesL} as for the harmonic Bargmann space.

\subsubsection{The Agmon-H\"ormander space}\label{ssAHspace}
The functions
\begin{equation*}
u_{k,j}(x)=Y_{k,j}(\o)|x|^{-\frac{d-2}{2}}J_{k+\frac{d-2}{2}}(|x|); \  k = 0,1,\dots, \ j=1,\dots, \db_k^\R,
\end{equation*}
form an orthogonal basis in the space $\Bc^{\Ab\Hb}$. The $\Ab\Hb$ norm of these functions equals $\frac1\pi$ (see, e.g., \cite{Strich}, p. 63). Thus, the eigenvalues of the Toeplitz operator $T_V$ in $\Bc^{\Ab\Hb}$ equal
\begin{equation}\label{EigenvAH}
    \L_k^{\Ab\Hb}=\pi \int_0^\infty V(r)J_{k+\frac{d-2}{2}}(|r|)^2 rdr.
\end{equation}
For  $V$ being the characteristic function  of $D_b$, these eigenvalues have the asymptotics
\begin{equation}\label{Eigenv.AH}
    \L_k^{\Ab\Hb}\sim\pi\left(\frac{b^2}{k}\right)^{k+\frac{d-2}{2}}(\G(k+\frac{d}{2}))^{-2}.
\end{equation}
So, the eigenvalues of the Toeplitz operator in the space $\Ab\Hb$ decay considerably faster than in the space $\Bc^\Hb$, with the same symbol.
Counting multiplicities, we obtain for the eigenvalues of $T_V$ the asymptotics
\begin{equation}\label{Eigenv.AH.log}
    \log \l_n=\log(s_n)\sim -\frac1{d-1}((d-1)!/2)^{\frac{2}{d-1}}(n^{\frac2{d-1}}\log n), \ n\to\infty
\end{equation}
and
\begin{equation}\label{Eigenv.AH.N}
    n(\l)\sim 2((d-1)!)^{-1}\left(\frac{|\log\l|}{\log|\log\l|}\right)^{\frac{d-1}2}.
\end{equation}

 \subsection{Sign-definite radial symbols with compact support}\label{compsupp}
 We introduce the following notion.
 \begin{definition}\label{exactSupport} Let the function $V(r), \ r\ge0,$ have compact support. The number $b$ is called the \emph{exact support radius}  (ESR) for $V$ if $V(r)=0$ for $r>b$,  while for any $b'\in(0,b)$,
 \begin{equation*}
    \int_{b'}^b|V(r)| dr >0.
 \end{equation*}
  \end{definition}

 \begin{proposition}\label{PropNonneg} Suppose that $b>0$  is  the ESR for $V\ge0$. Then for the operator $T_V$ in the spaces $\Bc_\Rb^\C, \Bc_\Rb^\R, \Bc_\Rb^\Hb, \Bc^\C, \Bc^\R, \Bc^\Hb, \Bc^{\Ab\Hb}$ hold the asymptotic formulas \eqref{asymp.compl}, respectively, \eqref{asymp.harm}, \eqref{asymp.harm}, \eqref{CBargm.eigenvalues}, \eqref{RBargm.eigenvalues}, \eqref{RBargm.eigenvalues}, and \eqref{Eigenv.AH.log} (as well as the corresponding asymptotic formulas for the counting function.)
 \end{proposition}
 \begin{proof} In all cases, the asymptotic estimate from above is trivial, since the numerator in the expression for the eigenvalues $\L_k$ increases when $V$ is replaced by the characteristic function of the ball with radius $b$, multiplied by some positive constant, and this constant is not felt in the logarithmic scale.  As for  the lower estimates, the reasoning is similar for all  cases. We present it, as an example, for  operator in the  space $\Bc_\Rb^\C$.

 For the operator $T_V$ in  $\Bc_\Rb^\C$,  fix some $b'<b.$   We have
 \begin{equation}\label{lower.posit.C.Be}
    \int_0^b V(r)r^{2k+2d-1}dr\ge \int_{b'}^b V(r) r^{2k+2d-1}dr\ge
    (b')^{2k+2d-1}\int_{b'}^b V(r)dr.
 \end{equation}
 Passing to the logarithmic scale, we obtain
 \begin{equation*}
    \log(\l_n^+)=\log(s_n(T_V))\gtrsim n^{1/d}(d!)^{\frac1d}\log (b'/\Rb),
 \end{equation*}
 which gives the required lower asymptotic estimate, due to the arbitrariness of $b'$.
 \end{proof}
 \subsection{Rapidly decaying sign-definite symbols}
\begin{definition}\label{RapidDecayDef} A bounded  function $V(r), \ r\in[0,\infty),$ is called rapidly decaying, $V\in\Rc\Dc,$ if
\begin{equation}\label{RD1}
    V(r)=o(\exp(-r^{\vs})), \ r\to\infty \mbox{ for any } \vs>0,
\end{equation}
or, equivalently, $\log |V(r)|<-Cr^\vs, C>0,$ for any $\vs>0.$
\end{definition}
For  further reference, we formulate here an important property of functions in $\Rc\Dc$, which is easily established by a proper change of variables.
\begin{lemma}\label{LemRD}If $V\in\Rc\Dc$, then
\begin{equation}\label{RD2}
 \left|\int_0^\infty V(r)r^s dr\right|=O(\G(\e s)), \ s\to\infty,
\end{equation}
for any $\e>0$.
\end{lemma}
\begin{proposition} Let $V\ge0$, $V\in \Rc\Dc$. Then for the eigenvalues of the operator $T_V$ in the spaces $\Bc^\C, \Bc^\R, \Bc^\Hb, \Bc^{\Ab\Hb}$ the  eigenvalue asymptotic formulas \eqref{CBargm.eigenvalues}, respectively, \eqref{RBargm.eigenvalues}, \eqref{RBargm.eigenvalues}, and \eqref{Eigenv.AH.log} hold (as well as the corresponding asymptotic formulas for the counting function.)
\end{proposition}
\begin{proof} Consider the complex Bargmann space first. If $V\in\Rc\Dc$, $V\ge0$, the numbers $\L_k^{\C}$ are given by the same formula \eqref{CBargm}. By monotonicity and Lemma \ref{LemRD},
\begin{equation}\label{RD3C}
   |\log  \left|\int_0^\infty V(r)r^k dr\right||=o(k\log k), \ k\to\infty,
\end{equation}
and therefore, for the numbers $\log \L_k^{\C}$ we have the same asymptotics \eqref{BargLambda} as for a compactly supported symbol, which leads to the  asymptotics
\eqref{CBargm.eigenvaluesN} for the counting function of the operator $T_V$.
The same reasoning takes care of the space $\Bc^\R$.

Consider now the operator in the space $\Bc^\Hb.$ For a general $V\in \Rc\Dc$, for estimating the numerator in \eqref{HarmBargm}, we use the representation formula for Bessel functions,
 \begin{equation}\label{BeselIntRepr}
    J_\n(r)=\left(\frac{r}2\right)^{\n}\left[\G(\n+1/2)\G(1/2)\right]^{-1}\int_{-1}^1 (1-t^2)^{\n-1/2}\cos(rt) dt, \ \re\n>-\frac12,
 \end{equation}
 see, e.g., \cite{Grad},  8.411.8.
It follows from \eqref{BeselIntRepr} that $|J_{k+\frac{d-2}{2}}(r)|\le C r^{k+\frac{d-2}{2}}\G(k+\frac{d-1}{2})^{-1}$; substituting this bound into \eqref{HarmBargm}, we obtain
\begin{equation}\label{HelmEigenv}
    \L_k\le C 2^k (\G(k+\frac{d+1}{2}))^{-1}\left[\int_0^\infty V(r) r^{k+d}dr\right]^2.
\end{equation}
By Lemma \ref{LemRD}, the integral in \eqref{HelmEigenv} is majorated by $\G(\e (k+d))$ for any $\e>0$. So, in logarithmic scale,
\begin{equation}\label{HelmEigenv1}
\log \L_k\sim-\log(\G(k+\frac{d+1}{2}))+o( k\log k),
\end{equation}
 which gives the same  asymptotics for $\log \L_k$, as in the case of $V$ with compact support.
So, for $V\in \Rc\Dc$ the same asymptotics \eqref{CBargm.eigenvalues}, \eqref{CBargm.eigenvaluesN} hold.

The same reasonings takes care of operators in the Agmon-H\"ormander spaces.
\end{proof}

\begin{remark}\label{RDRem} So, the asymptotic formulas for $n(\l)$ are the same for a compactly supported $V$ and for $V\in\Rc\Dc$. On the other hand, it was established in \cite{RaiWar} (see Theorem 2.1 there) that if a reasonably regular $V$ \emph{does not} belong to $\Rc\Dc$, the asymptotics of $n(\l; \Bc^\C)$ is different. This circumstance justifies the introduction of the class $\Rc\Dc$. Further on, in Section \ref{Sect5} we consider oscillating symbols not belonging to $\Rc\Dc.$
\end{remark}


 \subsection{Radial symbols, sign-definite at the periphery}
 As it was found in \cite{PushRoz2} for  Toeplitz operator in $\Bc^\C$ in dimension $d=1$, the asymptotics of eigenvalues is determined only by the sign of  $V$ at the periphery of its support. It turns out that such effect is present in other dimensions and other spaces as well. We explain the corresponding results for the case of a radial symbol, however, with a proper formulation, they hold also in much more general case.

 \begin{proposition}\label{PropPeriph} Suppose that  the ESR for the function $V(r)$ equals $b>0$, $0<b\le\infty$, and for some $b_0<b$,  $V(r)\ge0$ for $r\in(b_0,b).$ Then for such $V$ there are only finitely many negative eigenvalues and the assertion of Proposition \ref{PropNonneg} holds true.
 \end{proposition}
 \begin{proof}The proof follows  the ideas of Theorem 1.1 in \cite{PushRoz2}. All cases are treated in a similar way, so we consider only the operator in $\Bc^{\Hb}$  as an example.

 The upper asymptotic estimate is, again, trivial. For the lower estimate,
 fix $b'\in(b_0,b)$. Let $G_R(x,y)$ be the Green function for the Dirichlet problem in the ball $D_R:r<R$ for the Helmholtz equation. Such function exists as long as zero is not an eigenvalue of the  Helmholtz operator with Dirichlet boundary conditions  in $D_R$. Such exceptional values of $R$ form a discrete set, therefore we can find an interval $(b_1,b_2)\subset(b',b),$ such that $G_R$ exists for all $R\in(b_1,b_2)$.  By our condition, the interval $(b_1,b_2)$ can be also chosen in such way that $\int_{b_1}^{b_2}V(r)dr>0$.  Note that the function $G_R(x,y)$ is smooth for $x\ne y$.

 Let $u(x)$ be a solution of the Helmholtz equation in $\R^d$. For $x\in D_{b_0}$ and $R\in(b_1,b_2)$, the following integral representation is valid
 \begin{equation}\label{integrRepr}
    u(x)=\int_{\partial D_R}u(y)K(x,y;R)dS_R(y),
 \end{equation}
 where $K(x,y;R)=G_{\n(y)}(x,y)$ is the derivative of $G$ in the direction of the outer normal to $\partial D_R$ at the point $y\in \partial D_R$ and $dS_R(y)$ is the normalized surface measure on $\partial S_R(y)$.
 We multiply \eqref{integrRepr} by $V(R)R^{d-1}$ and integrate in $R\in(b_1,b_2)$. Thus we obtain the integral representation
 \begin{equation}\label{integrRepr2}
     u(x)= \left(\int_{b_1}^{b_2} V(R)R^{d-1} dR\right)^{-1}\int_{|y|\in(b_1,b_2)}K(x,y;R)u(y)V(|y|)dy.
 \end{equation}
 Since $K(x,y;R)$ is a smooth bounded function for $|x|\le b_0$ and $|y|\ge b_2$, the integral operator  \eqref{integrRepr2}  assigning the function $u(x)$, restricted to $D_{b_0} $ and considered as an element in $L^2(D_{b_0})$, to the same function considered as an element of $L^2(D_{b_0})$ with weight $V$, is compact.
 Therefore, the quadratic form
 \begin{equation}\label{QformPlusMinus}
  \ab_{b_0}[u]\equiv  \int_{D_{b_0}}V(|x|)|u(x)|^2dx
 \end{equation}
 is compact with respect to the quadratic form
 \begin{equation}\label{Qformplus}
    \int_{D_{b}\setminus D_{b_0}}V(|x|)|u(x)|^2dx,
 \end{equation}
 all forms, recall, being considered on the space of solutions of the Helmholtz equation.

Now we represent the quadratic forms ratio for the operator  $T_V$ as
\begin{equation}\label{qFormRatio}
    \frac{<T_V u,u>}{<u,u>}=\left[1+\frac{\int_{D_{b_0}}V|u|^2 dx}{\int_{|x|>b_0} V|u|^2 dx}\right]\frac{\int_{D_b}V_{b_0}|u|^2 dx}{<u,u>}, \ u\in\Bc^{\Hb},
\end{equation}
where $V_{b_0}(r)=0,\ r<b_0,$ $V_{b_0}(r)=V(r)$ otherwise.

Due to the compactness, explained above, for any $\e>0$, there exists a subspace $\Lc_\e\subset \Bc^{\Hb}$, having finite dimension $\k(\e)<\infty$, and such that
 \begin{equation}\label{qFormRatio2}
 \left|\frac{\int_{D_{b_0}}V|u|^2 dx}{\int_{|x|>b_0} V|u|^2 dx}\right|<\e\ {\rm{for\ }} u\in \Bc^{\Hb}, {\rm{ orthogonal\ to\ }} \Lc_\e.
\end{equation}
For $\e=1/2$, this means that there are no more than $\k({1/2})$ negative eigenvalues of $T_V$. Further on, by the variational principle, for the positive eigenvalues of $T_V$ the estimate holds
\begin{equation*}
    \l_n^+(T_V)\ge \frac12\l_{n+\k({1/2)}}(T_{V_{b_0}}).
\end{equation*}
Now the required lower estimate follows from Proposition \ref{PropNonneg} applied to the nonnegative symbol $V_{b_0}$.
 \end{proof}
 \section{Auxiliary theorems}\label{Sect3}

\subsection{Re-ordering -- 2}\label{ss.Reordering2}
As it was explained in the  Introduction, the main complication for proving lower estimates for eigenvalues lies in the need of reordering of the sequence of eigenvalues $\L_k$ obtained by the explicit formulas in Sect. \ref{Sect2}. So, supposing that the lower estimate is wrong, and thus a contradicting upper estimate holds, we can obtain a bound for the re-ordered sequence of the numbers $\L_k$, which, however,  does not imply directly any estimate for the numbers $\L_k$ themselves. In order to deal with this circumstance, we need
the following statement  which plays a key role in the sequel.
\begin{proposition}\label{Prop.reord.2} Let $k\mapsto m_k$ be a bijection of the set of nonnegative integers $\Z_+$. For $\b>1$,  we denote by $E_\b$ the set $\{k\in\Z_+:m_k\le\b k\}$, $F_\b=\{m_k: k\in E_\b.\}$. Then
\begin{equation}\label{reorder}
    \#\{F_\b\cap[0,N]\}\ge \frac{\b -1}{\b} N,
\end{equation}
for any natural $N.$
\end{proposition}
In other words, the Proposition states that, under a bijection, a controllably nonzero share of integers $m_k$  are not too large, compared with $k.$
\begin{proof}
Suppose that $m_k\in[0,N]\setminus F_\b$. Then $k<m_k/\beta\le N/\beta$. And therefore $\#\{F_\b\cap[0,N]\}> N+1-N/\beta\ge \frac{\b-1}{\b} (N+1) -\frac{1}\b.$
\end{proof}
\begin{remark}\label{RemReorder2} The constant $\frac{\b-1}{\b}$  in Proposition \ref{Prop.reord.2} is sharp. In fact, set    $m_k=[\beta k]+1$ for $k\in\mathbb N\setminus2^{\mathbb N}$, while for  integers   powers of $2$ we define $m_k$  so that to obtain a bijection. Then $\limsup N^{-1}\#\{F_\b\cap[0,N]\}=\frac{\b-1}{\b}.$
\end{remark}
Proposition \ref{Prop.reord.2} leads to the following partial conversion of Proposition \ref{PropReorderTriv}, mentioned in Sect. \ref{SSreordering1}.
\begin{proposition}\label{prop.reord.major} Suppose that  $a_k, b_k$ are real sequences, $b_k>0$ is non-increasing, and for the non-increasing permutation $a^*_k$ of $|a_k|$, we have
\begin{equation}\label{reord.major}
    a^*_k\le b_k.
\end{equation}
Then for any $\b>1$ there exists a subsequence $a_{k_l}$, $k_1<k_2<\dots$, such that
\begin{equation}\label{reord.major.1}
    |a_{k_l}|\le b_{[k_l/\b]}
\end{equation}
and $k_l\le [\frac{\b}{\b-1}l]+1$.
\end{proposition}
The statement means that if the  sequence $a_k$, after being non-increasingly reordered, satisfies some sort of monotonous estimate, then in the initial sequence there exists a controllably dense subsequence, for which  a similar but slightly weaker estimate holds.
\begin{proof}Let the non-increasing permutation of the sequence $|a_k|$ be given by the bijection $j\mapsto m_j:$ $|a_{m_j}|=a^*_j,$  so that $|a_{m_j}|\le b_j$.
 Thus, for any $m_j\in F_\b$, we have $m_j\le j\b $, therefore $j\ge m_j/\b$ and $b_j\le b_{[m_j/\b]}$.
 Now we take as the subsequence $k_l$, the elements $m_j\in F_\b$ taken in the increasing order. The inequality \eqref{reord.major.1} is therefore fulfilled.
 By Proposition \ref{Prop.reord.2}, $\#\{F_\b\cap[0,N]\} \ge \frac{\b}{\b-1} N$ for any $N$, which is equivalent to  the second inequality we need.
\end{proof}

We will also need a simple consequence of Proposition \ref{Prop.reord.2} concerning the rate of divergence of the series composed of the inverse values of  $m\in F_\b$.
\begin{proposition}\label{PropInverse}
Under the conditions of Proposition \ref{Prop.reord.2},
\begin{equation}\label{inversesum}
    \limsup_{N\to\infty}{(\log N)^{-1}}{\sum\limits_{m\in F_\b\cap [0,N]}m^{-1}}\ge\frac{\b-1}{\b}.
\end{equation}
\end{proposition}

\subsection{Estimates for   functions analytical in a half-plane}
There are a number of results in the classical complex analysis relating the estimates along the real axis of a function analytical in the half-plane $\re \z >0, \z=\x+i\y$, with estimates of its values at some sequence of points. The first of such results we need, with ideas originating in \cite{Levinson}, was obtained in \cite{Boas}, p. 200.
\begin{theorem}\label{ThBoas}
Let $f(\z)$ be a function, analytical in the right half-plane, of exponential type,  satisfying
\begin{equation}\label{BoasCond}
    \int_{-\infty}^\infty \frac{\log_+|f(i\y)|}{1+\y^2}d\y<\infty.
\end{equation}
Suppose that $\m_l$ is a monotone sequence of real points tending to infinity so that $|\m_l-\m_{l-1}|\ge\d>0$ and  $\sum \m_l^{-1}=\infty.$
Then
\begin{equation}\label{BoasResult}
    \limsup_{l\to\infty}\frac{\log |f(\m_l)|}{\m_l}=\limsup_{\x\to+\infty}\frac{\log|f(\x)|}{\x}.
\end{equation}
\end{theorem}
\begin{remark}\label{BoasRemark} In \cite{Levinson} and in \cite{Boas} the additional condition
 $\lim l \m_l^{-1}=0$ was imposed. However,  it was shown  in \cite{Leont} that this condition is excessive and can be deleted.
 \end{remark}
 Theorem \ref{ThBoas} will be used for the study of the spectrum of operators in Bergman spaces. For the case of Bargmann spaces another result about estimates of functions, \emph{not} of exponential type,  will be used. We cite its version from \cite{Eiderman}, see also \cite{EidermanEssen}, with an obvious typo corrected.
 \begin{theorem}\label{thmEider}Let the function $f(\z)$ be analytical in the half-plane $\re\z=\x>0$ and satisfy the estimate

 \begin{equation}\label{Cond.Eider1}
    |f(\r e^{i\vf})|=O(\exp[\r(a\log \r \cos\vf+\pi c |\sin \vf|+b\cos \vf)]), \ |\vf|<\pi/2, \r\to\infty,
 \end{equation}
 for some $a\ge 0, c\ge -a/2.$ Suppose also that the  growing sequence $\m_l$ of positive numbers satisfies $\m_{l+1}-\m_l\ge \d>0$ and
 \begin{equation}\label{Cond.Eider2}
    \limsup_{N\to\infty}\left[\sum\limits_{\m_l\le N}\m_l^{-1}-(c+a/2)\log N\right]=\infty.
 \end{equation}
 Then the bound
 \begin{equation}\label{Cond.Eider3}
    \limsup_{l\to\infty}\frac{\log|f(\m_l)|}{\m_l\log\m_l}<-2c
 \end{equation}
 implies that $f(\z)\equiv 0.$
 \end{theorem}
 This theorem improves the classical result by N. Levinson, see \cite{Levinson}, Theorem XLI, in the sense that  it does not require any regularity of the sequence $\m_l.$
 \section{Eigenvalues of Toeplitz operators with non-sign-definite weight}\label{Sect4}
 This Section contains the main results of the paper. These results can be expressed in the following way:
 \begin{equation}\label{EstimGeneral}
    \limsup_{\l\to 0}\frac{n(\l; T_V)}{n(\l;T_{|V|})}=1,
 \end{equation}
 in all spaces under consideration, where  the  radial function $V$ has compact support for $\O=D_\Rb$, and $V\in\Rc\Dc$ for the case of $\O=\C^d$ or $\O=\R^d$. Further on, we consider the concrete cases in detail.
\subsection{Operators in Bergman spaces}
In this section the symbol $V$ is supposed to be an arbitrary real bounded radial function with ESR $b<\Rb.$
 \begin{theorem}\label{ThMainBergmanC} For the singular numbers of the operator $T_V$ in the Bergman spaces  the following asymptotic formulas hold.\\
 For the complex space  $\Bc^\C_{\Rb},$
 \begin{equation}\label{MainBergmanC1}
  \limsup_{\l\to0} n(\l)|\log\l|^{-d}= (d!)^{-1}(2|\log(b/\Rb)|)^{-d};\ n(\l)=n(\l; T_V,\Bc^\C_\Rb).
 \end{equation}
 For the harmonic and Helmholtz spaces $\Bc_\Rb^\R$ and $\Bc_\Rb^\Hb:$
 \begin{gather}\label{MainBergmanR1}
    \limsup_{\l\to0} n(\l)|\log\l|^{-d}=2 ((d-1)!)^{-1}(2|\log(b/\Rb)|)^{-d+1};\\
  {\rm for}\  n(\l)=n(\l; T_V,\Bc^\C_\Rb) \ {\rm{or }} \ n(\l)=n(\l; T_V,\Bc^\Hb_\Rb).\nonumber
 \end{gather}

 \end{theorem}
 \begin{proof} The proofs for the analytical and harmonic cases are almost identical; we present the first one.

 The upper estimate in \eqref{MainBergmanC1} follows from Proposition \ref{PropNonneg} by monotonicity and Proposition \ref{PropReorderTriv}. We will prove the lower estimate. Suppose that it is wrong; this means that
 \begin{equation}\label{MainBergmanC22}
    n(\l; T_V,\Bc^\C_\Rb)<\g (d!)^{-1}(2|\log(b/\Rb)|)^{-d}|\log\l|^{d}
 \end{equation}
for some  $\g<1$ and
 for $\l $ small enough.

 The singular numbers of $T_V$ are  equal to the numbers $|\L_k|$ defined in \eqref{Bergman C}, permuted in the non-increasing order (we denote by $\s_m$ this permuted sequence), with multiplicities $\db_k^\C$ given in Sect.\ref{Sect2}.  So,  $\s_{m_k} =|\L_{k}|$, where $k\mapsto m_k$ is some bijection of $\Z_+$.

By \eqref{asymp.compl.count},
\begin{equation}\label{MainBergmanC3}
    n(\l)=\sum_{|\L_k|>\l}\db_k^{\C}=\sum_{\s_{m_k}>\l}\db_k^{\C}.
\end{equation}
Since the numbers $\db_k^{\C}$  increase with $k$ growing,  the quantity in \eqref{MainBergmanC3} can only decrease if we replace in \eqref{MainBergmanC3} the values of $k$ by their smallest possible values, i.e.,
 \begin{equation}\label{MainBergmanC4}
     n(\l)\ge \sum_{k=0}^{n_0(\l)}\db_k^{\C},
 \end{equation}
 where ${n_0(\l)}=\#\{j:\s_j>\l\}.$

 So, since $\db_k^{\C}=(1+O(k^{-1}))k^{d-1}((d-1)!)^{-1}$ for large $k$, we have
 \begin{equation}\label{MainBergmanC5}
    n(\l)\ge \sum_{k=0}^{n_0(\l)} \frac{k^{d-1}(1+O(k^{-1}))}{(d-1)!}\ge(1+O(n_0(\l)^{-1}))\frac{n_0(\l)^d}{d!}.
 \end{equation}

 Substituting \eqref{MainBergmanC5} into \eqref{MainBergmanC22}, we obtain
 \begin{equation}\label{MainBergmanC6}
    n_0(\l)\le (1+O(|\log \l|^{-1}))\g^{\frac1d}|\log \l|, \ \mbox{or} \s_m\lesssim \left(\g'\frac{b}{\Rb}\right)^{2m},
 \end{equation}
 for $\l$ small enough, resp.,  $m$  large enough and some $\g'<1$. Our next aim is to derive an estimate for $\L_k$ from \eqref{MainBergmanC6}.

 We fix some $\b>1$, to be determined  later,  and apply Proposition \ref{prop.reord.major}  to the sequences $a_k=|\L_k|$, $b_k=\left(\g'\frac{\b}{\Rb}\right)^{2m}.$ Thus there exists a subsequence $\L_{k_l}$ such that
 \begin{equation}\label{MainBergmanC7}
    |\L_{k_l}|\le  C\left(\g'\frac{b}{\Rb}\right)^{{2}k_l\b^{-1}}.
 \end{equation}

We are going to show now  that the inequality   \eqref{MainBergmanC7} holds not only for the subsequence $\L_{k_l}$ but for all $\L_k$, probably, with  slightly worse constants. To do this, we introduce the complex variable $\z=\x+i\y$ and consider the function

 \begin{equation}\label{MainBergmanC8}
    f(\z)=(2\z+2d)R^{-(2\z+2d)}\int_{0}^{\Rb}V(r)r^{2\z+2d-1}dr.
 \end{equation}

 The function $f(\z)$ is analytical and bounded in the half-plane $\re\z=\x>0$, so the condition \eqref{BoasCond} is satisfied.  The values of $f$ at integer points  $k$ coincide with the numbers $\L_k$, due to \eqref{Bergman C}.
 By the second inequality in \eqref{MainBergmanC6}, the series $\sum (k_l)^{-1}$ diverges. So, all conditions or Theorem \ref{ThBoas} are fulfilled and, therefore,
 \begin{equation*}
    \limsup_{k\to\infty}\frac{\log |\L_{k}|}{k}= \limsup_{l\to\infty}\frac{\log|\L_{k_l}|}{k_l},
 \end{equation*}
 or, returning back from the logarithmic scale,
 \begin{equation}\label{MainBergmanC9}
    \L_k\le C\left(\g'\frac{b}{\Rb}\right)^{{2}k/\b'}
 \end{equation}
 for any $\b'>\b$. Since $\g'<1$, we can choose   the parameter $\b$ and then $\b'>1$ in the above reasoning so close to $1$ that
 \begin{equation}\label{MainBergmanC10}
    \left(\g'\frac{b}{\Rb}\right)^{{2/\b'}}< \left(\th\frac{b}{\Rb}\right)^{{2}}
 \end{equation}
  for some $\th<1.$ We substitute \eqref{Bergman C} into \eqref{MainBergmanC9} and obtain
 \begin{equation}\label{MainBergmanC11}
    \int\limits_0^\Rb V(r)r^{2k+2d-1}dr =O((\th b)^{2k}).
 \end{equation}
 It remains to apply a classical theorem about the properties of the moments problem, say, Theorem 6.9.5 in \cite{Boas}, saying that \eqref{MainBergmanC11} implies
 $\supp V\subset[0,\th b]$, and this inclusion contradicts our condition that $b$ is the ESR for $V$.

 Now we pass to the proof for the Helmholtz case. Again, suppose that
 \eqref{MainBergmanR1} is wrong; this means that
 \begin{equation}\label{MainBergmanH2}
      n(\l; T_V,\Bc^\Hb_\Rb)|\log\l|^{-d+1}<2\g ((d-1)!)^{-1}(2|\log(b/\Rb)|)^{-d+1}
 \end{equation}
 with some $\g<1$, for small $\l.$ In the same way as for the case of the complex spaces, \eqref{MainBergmanH2} implies the estimate for the numbers $\s_m$, the monotonically re-ordered sequence of the numbers $\L_k=\L_k^\Hb:$

  \begin{equation}\label{MainBergmanH3}
  \s_m\le \left(\g\frac{b}{\Rb}\right)^{2m}.
  \end{equation}
   We can further proceed  as before, to derive from \eqref{MainBergmanH3} the estimate for a sufficiently dense subsequence in $\L_m:$
   \begin{equation}\label{MainBergmanH4}
     |\L_{k_l}|\le  C\left(\g\frac{b}{\Rb}\right)^{{2}k_l/\b'},
   \end{equation}
   with $k_l\le \frac{\b}{\b-1}l.$

Next, as before, we need to carry over the estimate \eqref{MainBergmanH4} from the subsequence to the whole sequence $\L_k$, and even to the fractional $k$. To achieve this, we consider the auxiliary function $f(\z)$ analytical in the half-plane $\x=\re\z>0:$
\begin{equation}\label{MainBergmanH5}
    f(\z)=\G(\z+d/2)\G(\z+(d+2)/2)){\int_0^\Rb V(r)J^2_{\z+\frac{d-2}{2}}(r)rdr}.
\end{equation}
By the known asymptotics of Bessel functions for the large index (see \eqref{BesselAs}),  the function $f(\z) $ is bounded in the half-plane $\re\z>0$,  moreover,   by \eqref{EigenHelm}, \eqref{BesselIntegr}, and \eqref{BesselintAs}, its values  the integer points $\z=k$ are asymptotically equal to the numbers $\L_k$.

 Therefore, we can apply Theorem \ref{ThBoas} to the  function $f(\z)$, similarly to the reasoning for the complex case above,  thus obtaining  for real $\x>0$ and some $\g'\in{(\g,1)}$
 \begin{equation}\label{MainBergmanH4a}
    |f(\x)|\le  C\left(\g'\frac{b}{\Rb}\right)^{{2}\x/\b}, \ \x\to\infty,
   \end{equation}
   which means
   \begin{equation}\label{MainBergmanH5aa}
    \int_0^\Rb V(r)J^2_{\x+\frac{d-2}{2}}(r)rdr\le  C[\G(\x+d/2)\G(\x+(d+2)/2)]^{-1}\left(\g'\frac{b}{\Rb}\right)^{{2\x/\b}}, \ \x\to\infty.
   \end{equation}
   Now we need to consider the cases of even and odd $d$ separately.
   
   For the case of an even dimension $d$, we write \eqref{MainBergmanH5aa} for integer values $\x=k:$
   \begin{equation}\label{MainBergmanH5a}
    \int_0^\Rb V(r)J_{k+\frac{d-2}{2}}(r)^2rdr\le  C[\G(\k+d/2)\G(k+(d+2)/2)]^{-1}\left(\g'\frac{b}{\Rb}\right)^{{2k/\b}}.
   \end{equation}
   We use now  C.Neumann's  formula, see \cite{Watson},  2.72(2),  or \cite{Grad}, 8.536.2:
   \begin{equation}\label{Besselseries}
  \sum_{j=m}^{\infty}\frac{\G(m+j)}{j\G(j-m+1)}J_j^2(r)=\frac{(2m)!}{(m!)^2}\left(\frac{r}2\right)^{2m}, \ m\in\Z_+.
   \end{equation}
  It is easy to see from the Stirling formula that the series in \eqref{Besselseries} converges uniformly on finite intervals, and, asymptotically in $m$, the leading term prevails.  We substitute the expression for $r^{2m}$ with $m=k+\frac{d-2}{2}$  from \eqref{Besselseries}   into $\int_0^\Rb r^{2(k+\frac{d-2}{2})}V(r)dr:$
  \begin{equation}\label{MainBergmanH6}
    \int_0^\Rb r^{2m}V(r)rdr=2^{2m}\frac{(m!)^2}{(2m)!} \sum_{j=m}^{\infty}\frac{j\G(m+j)}{\G(j-m+1)}\int_0^\Rb V(r)J^2_j(r)r dr.
  \end{equation}
  For each term in \eqref{MainBergmanH6}, we apply the estimate \eqref{MainBergmanH5a}. Calculations with $\G$-functions show that
  \begin{equation}\label{MainBergmanH7}
    \int_0^\Rb r^{2k}V(r)r^{d-1}dr=O(\left(\g'{b}\right)^{{2k/\b}}).
  \end{equation}
  Finally, since $\g'<1$, we can choose $\b>1$ so close to $1$ that $ \left(\g'{b}\right)^{1/{\b}}<\th b, \ \th<1$. Now we can again apply Theorem 6.9.5 in \cite{Boas}, which gives $\supp V\subset[0,\th b]$, $\th<1$ which contradicts our assumption that $b$ is the ESR for $V$.

  A similar reasoning works for case of an odd dimension $d.$ We however consider   \eqref{MainBergmanH5aa}  for half-integer $\x$, $\x=k+1/2$.  Then, since $(d-2)/2$ is half-integer,  \eqref{MainBergmanH5aa} gives an estimate of integrals containing Bessel functions with integer index.
 The final step in the proof is the same.
\end{proof}
\subsection{Operators in Bargmann and AH-spaces}
For the case of Bargmann and AH spaces, the consideration follows the same idea as for the Bergman spaces, with minor modifications.
\begin{theorem}  \label{ThEstBargC}  Let the radial symbol $V$ belong to $\Rc\Dc.$ Then
(i) For $\Bc=\Bc^\C(\C^d)$,
\begin{equation}\label{SnizuCompl}\limsup_{\l\to 0}\left[n(\l) \left(\frac{|\log\l|}{\log|\log\l|}\right)^{-d}\right]=  (d!)^{-1} .
\end{equation}
(ii) For $\Bc=\Bc^\R(\R^d)$, or  $\Bc^\Hb(\R^d)$
\begin{equation}\label{SnizuReal}
    \limsup_{\l\to 0}\left[n(\l) \left(\frac{|\log\l|}{\log|\log\l|}\right)^{-d+1}\right]= ((d-1)!)^{-1}.
\end{equation}
(iii) For $Bc=\Bc^{\Ab\Hb}(\R^d)$
\begin{equation}\label{SnizuAH}
\limsup_{\l\to 0}\left[n(\l) \left(\frac{|\log\l|}{\log|\log\l|}\right)^{-d+1}\right]=2 ((d-1)!)^{-1}.
\end{equation}
\end{theorem}

\begin{proof} All cases are proved in a similar manner.
We give the proof of the part (i) of Theorem \ref{ThEstBargC} in detail and then explain the changes needed for other cases.  

The proof starts in the same way, as for Theorem \ref{ThMainBergmanC}.  The estimate from above in \eqref{SnizuCompl} is already established. Suppose that the lower bound in  \eqref{SnizuCompl} is wrong. This means that for some $\g<1$, the inequality
\begin{equation}\label{SnizuC.1}
n(\l)<\g (d!)^{-1}\left(\frac{|\log\l|}{\log|\log\l|}\right)^{d}
\end{equation}
holds for all sufficiently small $\l>0$.
The singular numbers of the operator $T_V$ are  equal to  the numbers $|\L_k|$, see \eqref{CBargm}, permuted in the non-increasing order, with multiplicities $\db^\C_k $ (defined in Sect. 2.3.1). So, these singular numbers equal $\s_{m_k}=|\L_k|$, where $k\mapsto m_k$ is some bijection of $\Z_+.$
By \eqref{counting multi},
\begin{equation}\label{SnizuC.2}
    n(\l)=\sum_{|\L_k|>\l}\db^\C_k=\sum_{\s_{m_k}>\l}\db^\C_k.
\end{equation}
Since the multiplicities $\db^\C_k$ are non-decreasing as $k$ grows, the quantity \eqref{SnizuC.2} can only decrease if we replace in the sum in \eqref{SnizuC.2} the values of $k$ by their smallest possible values,
\begin{equation}\label{SnizuC.3}
    n(\l)\ge \sum_{k=0}^{n_0(\l)}\db^\C_k,
\end{equation}
where $n_0(\l)=\#\{\s_m>\l\}.$
So, since $\db^\C_k= k^{d-1}((d-1)!)^{-1}(1+O(k^{-1}))$, we have
\begin{equation}\label{SnizuC.4}
    n(\l)\ge\sum_{k=0}^{n_0(\l)}\frac{k^d}{(d-1)!}(1+O(k^{-1}))=(1+O(n_0(\l)^{-1}))\frac{n_0(\l)^d}{d!}.
\end{equation}
We substitute \eqref{SnizuC.4} into \eqref{SnizuC.1} and obtain
\begin{equation}\label{SnizuC.5}
  \g'n_0(\l)\le (1+o(1)) \frac{|\log\l|}{\log|\log\l|}, \  \g'\in (1, \g^{-\frac1d}).
\end{equation}

We rewrite   \eqref{SnizuC.5} in   terms of an estimate for $\s_m:$
\begin{equation}\label{SnizuC.6}
\s_m\lesssim \G(\g' m)^{-1}.
\end{equation}
We   derive now an estimate for $\log|\L_k|$ from \eqref{SnizuC.6}.
Fix some $\b>1$, to be  determined later, and apply Proposition \ref{prop.reord.major} to the sequences $a_k=|\L_k|, b_k=\G(\g' m_k)^{-1}.$ Thus there exists a subsequence $\L_{k_l}$ such that

\begin{equation}\label{SnizuC.7a}
    |\L_{k_l}|\lesssim \G(\g' k_l/\b)^{-1},
\end{equation}
or, in the logarithmic scale,
\begin{equation}\label{SnizuC.7}
    \log|\L_{k_l}|\le -\g'\b^{-1} k_l \log ( k_l/\b)(1+o(1)).
\end{equation}

Now we introduce the complex variable $\z=\x+i\y$ and consider  in the half-plane $\x>0$ the function
\begin{equation}\label{SnizuC.8}
    f(\z)=2{\int_0^\infty V(r)r^{2\z+2d-1} e^{-r^2}dr}.
\end{equation}
This function  coincides with $\L_k{\G(k+d)}$ at the  integer points $\z=k$. Therefore, by \eqref{SnizuC.7} at the points $\z=k_l$, the function $ f(\z)$ satisfies
\begin{equation}\label{SnizuC.80}
   \log|f(k_l)|<\log\G( k_l/\b)-\g'\b^{-1} k_l \log ( k_l/\b)(1+o(1))\lesssim(1-\g'/\b)\log\G(k_l/\b).
\end{equation}

 By Lemma \ref{LemRD},  for $\z=\r e^{i\vf},$
\begin{equation}\label{SnizuC81}
    |f(\z)|=O\left(\int_0^\infty |V(r)|r^{2\r\cos\vf+2d-1} e^{-r^2}dr\right)=O(\G(\e\r\cos\vf)), \ \r \to +\infty
\end{equation}
for any $\e>0. $ Taking into account the asymptotics for the $\G$-function for large real values of argument, we obtain that the estimate \eqref{Cond.Eider1} is satisfied for any $a>0, c>0.$
Now, we choose  $\b>1$ so that $\g'/\b>1.$ After this, we fix $a,c>0$ so small that $\frac{\b-1}\b>a+c/2$ and $2c<\g'/\b-1$. Then, by Proposition \ref{PropInverse}, the sequence of integers  $\{\m_l\}$ obtained by the increasing reordering of the sequence $\{k_l\}$, satisfies the conditions \eqref{Cond.Eider2} and \eqref{Cond.Eider3}. So, all conditions of Theorem \ref{thmEider} are satisfied, and we can conclude that $f(\z)=0$ for all $\z.$

It remain to notice that $f(-\z)$ is the Mellin transform of the function $V(r)r^{2d-1}e^{-r^2}$, and by the inversion theorem for the Mellin transform we conclude that $V(r)\equiv0$, and this takes care of the proof  for the complex Bargmann space.

The proof for the space $\Bc^\R$ follows the reasoning above, only with $d$ replaced by $d-1$.

Now we consider the operator in the Helmholtz Bargmann space, where some  more changes are needed.

Similar to \eqref{SnizuC.1}, we suppose that \eqref{SnizuReal} is wrong, and this would mean that $n(\l)\le \g((d-1)!)^{-1}\left(\frac{|\log\l|}{\log|\log\l}\right)^{d-1}$
for small $\l$ and some $\g<1$. By repeating the calculations in \eqref{SnizuC.2}-\eqref{SnizuC.4}, we obtain for the distribution function $n_0(\l)$ (now, of the numbers $\L_k=\L_k^{\Hb}$) an estimate of the form \eqref{SnizuC.5}, with some $\g'>1.$ After this, we  apply again Proposition \ref{prop.reord.major} to obtain for a (sufficiently dense) subsequence $\L_{k_l}$ the estimate \eqref{SnizuC.7}.  Starting from this point, the reasoning is somewhat different.

We introduce the complex variable $\z=\x+i\y,\  \x\ge0,$ and consider the  function
\begin{equation}\label{sniHelm1}
    f(\z)=\G^2\left(\z+({d-2})/{2}\right)\int_0^\infty J^2_{\z+\frac{d-2}{2}}V(r)e^{-r^2}rdr.
\end{equation}
This function is analytical in the half-plane $\x>-\frac12$. It follows also  from \eqref{BeselIntRepr} that
\begin{equation}\label{sniHelm2}
    |  f(\z)|\le C\int_0^\infty r^{2(\z+d/{2})}|V(r)|e^{-r^2}dr,
\end{equation}
and by \eqref{RD2},   $f(\z)$ satisfies
\begin{equation}\label{sniHelm3}
   |f(\z)|=O(|\G(\e \z)|)
\end{equation}
for any $\e>0$ as $|\z|\to \infty$, $\re\z>-\frac12$.  It follows that the function $f(\z)$ satisfies condition \eqref{Cond.Eider1} of Theorem \ref{thmEider}, with arbitrarily small positive values of $a,c$.

After this property has been established, the proof follows the one for the  complex Bargmann space.
For real integer $\z=k$ we have $|f(k)|\asymp|\L_{k}| \G(k+(d-2)/2)$. Therefore, by \eqref{SnizuC.5}, the values of $f(\z) $
 at the points $\z=k_l$ satisfy \eqref{SnizuC.80}. We take $\b>1$ so close to $1$ that $\g'/\b>1$ and then fix $a,c$ so small that $\frac{\b-1}\b>a+c/2$ and $2c<\g'/\b-1$.  Then the sequence of integers $\m_l$ obtained by the increasing re-ordering of the sequence $k_l$, satisfies all conditions of Theorem \ref{thmEider}, and therefore $f(\z)\equiv 0.$  This means, in particular, that all numbers $\L_k$ are zeros, therefore $T_V=0.$     To prove that this implies $V=0$, we use  \eqref{Besselseries} to express $\int_0^\infty r^{2m} V(r)e^{-r^2} dr$ as a linear combination of the numbers $\L_k$, all of them being  equal to zero. Therefore we obtain that the function $g(\z)=\int_0^\infty r^{2\z+2d} V(r)e^{-r^2} dr$ takes zero values at all integer points. Finally, we note that, again by Lemma \ref{LemRD}, function $g(\z)$ satisfies the conditions of Theorem \ref{thmEider} with arbitrarily small positive $a,c$, and therefore $g\equiv 0$. The proof  that $V\equiv 0$ concludes again by the inversion theorem for the Mellin transform.

The case of the space $\Bc^{\Ab\Hb}$
 is proved in the same way, with minimal changes..
\end{proof}
\section{Not that rapidly decaying symbols. A counterexample}\label{Sect5}
The results of Sect.\ref{Sect4} might lead to the impression that, probably, one should expect  the absence of cancelation of the positive and negative  parts of the symbol in a more general situation as well. The author was of such opinion for a certain time. However, the example given in this Section shows that such impression is wrong.

We present here a construction of symbols that decay at infinity  rather fast but not sufficiently fast to get into the $\Rc\Dc$ class. This symbols oscillate very rapidly at infinity.  We show that the s-numbers of the Toeplitz operator with symbol $V$ decay essentially faster than the ones for the operator with symbol $|V|,$ so an analogy with theorems in Sect. \ref{Sect4} does not hold.  In order to simplify the calculations, we restrict ourselves here to the operators in the space $\Bc^\C$ in the one-dimensional case, $d=1.$  The same constructions work in any dimension, and for the spaces $\Bc^\R$ as well. We suppose that for other spaces under consideration a similar construction produces proper examples.

We consider the symbol 
\begin{equation}\label{Ncounter1}
   V(r)= V_{p,q}(r)=e^{-r^{2p}+r^2}\sin( r^{2q})
\end{equation}
with $p>1$, $q>p$.

\begin{theorem}\label{ThmNcounter}
For the operator $T_V$
\begin{equation}\label{NCounter2}
    \limsup_{\l\to 0}n(\l, T_V) \frac{\log|\log\l|}{|\log \l|}\le\frac{q}{q-1},
\end{equation}
while
\begin{equation}\label{NCounter3}
     n(\l, T_{|V|})\sim n(\l, T_{V_+})\sim n(\l,T_{V_+})\sim\frac{p}{p-1}\frac{|\log \l|}{\log|\log\l|}, \ \l\to 0.
\end{equation}
\end{theorem}

The difference in coefficients in front of $\frac{|\log \l|}{\log|\log\l|}$ in \eqref{NCounter2}, \eqref{NCounter3} transforms into a large difference in the decay order of the eigenvalues:
\begin{equation}\label{NCounterLog}
    \log s_n(T_V)\lesssim \frac{q(p-1)}{p(q-1)}\log s_n(T_{|V|}),
\end{equation}
so, in fact,  a rather strong cancelation takes place.
\begin{proof}
In order to prove \eqref{NCounter3}, we need to estimate the numerator in \eqref{CBargm}, i.e., the integral
\begin{equation}\label{NCounter4}
    \Ib(k)=\int_0^\infty e^{-r^{2p}+r^2}\sin\left(r^{2q}\right) e^{-r^2}r^{2k+1}dr.
\end{equation}
After the change of variables $t=r^{2q}$, \eqref{NCounter4} transforms to
\begin{equation}\label{NCounter5}
     \Ib(k)=(2q+1)^{-1}\int_0^\infty t^{\frac{k+2}{q}-1}e^{-t^{p/q}}\sin t dt=\Im (2q+1)^{-1}\int_0^\infty t^{\frac{k+2}{q}-1}e^{-t^{p/q}}\exp(it) dt .
\end{equation}
We consider here $t$ as complex variable living on the positive real half-line and the whole expression \eqref{NCounter5} as the integral in along this half-line. 

Now we replace the integration line by means of rotating it to   the line $\arg t=\pi/2$. The integral in \eqref{NCounter5} does not change due to the factor $e^{-t^{p/q}}$ which decays fast in the whole first quarter.
So, by setting $t=i\t$, $\t\in(0,1)$, we have
\begin{equation}\label{NCounter6}
    |\Ib(k)|\le (2q+1)^{-1}\int_0^\infty \t^{\frac{k+2}{q}-1}e^{-\t} d\t =\G((k+2)/q).
\end{equation}
Finally, taking into account the  expression for the denominator in \eqref{CBargm} and Stirling's formula  we arrive at \eqref{NCounter2}.

To prove \eqref{NCounter3}, we estimate the integral, say,
$$\Ib_+(k)=\int_0^\infty t^{\frac{k+2}{q}-1}e^{-t^{p/q}}\sin_+(t) dt$$
from below. To do this we consider  the intervals $I_j=(2\pi/3+4\pi j, 4\pi/3+4\pi j)$. On each of these intervals, for $j$ large enough, $\sin(t)>\frac12$. This inequality easily implies a lower estimate for $\Ib_+(k):$
$$\Ib_+(k)\ge\sum_j \frac12 \int_{I_j}{\frac{k+2}{q}-1}e^{-t^{p/q}}\sin_+(t) dt \ge C\int _0^\infty t^{\frac{k+2}{q}-1}e^{-t^{p/q}}dt=C \G\left(\frac{k+2}{p}\right),$$
which leads to \eqref{NCounter3}.
\end{proof}

\end{document}